\documentclass[reqno, oneside, english]{amsart}

\usepackage[letterpaper]{geometry}
\usepackage{amsmath}
\usepackage{amssymb}
\usepackage{enumerate}
\usepackage{enumitem}
\usepackage{verbatim}
\usepackage{mathrsfs}
\usepackage{mathtools}

\mathtoolsset{centercolon}

\usepackage{tikz}
\newcommand*\circled[1]{\tikz[baseline=(char.base)]{
            \node[shape=circle,draw,inner sep=1.2pt] (char) {#1};}}

\usepackage{setspace}
\numberwithin{equation}{section}

\newtheorem{thm}{Theorem}[section]

\newtheorem{lem}[thm]{Lemma}
\newtheorem{prop}[thm]{Proposition}

\theoremstyle{remark}

\theoremstyle{definition}

\newcommand{\Z}{\mathbb{Z}}

\newcommand{\SL}{\text{SL}_{2}(\Z)}

\newcommand{\Shim}{\mathscr{S}}

\renewcommand{\t}[2]{T_{{#1}}(#2)}
\renewcommand{\tt}[3]{T_{#1}(#2,#3)}
\newcommand{\T}[2]{\big|T_{#1}(#2)}
\newcommand{\TT}[3]{\big|T_{#1}(#2,#3)}
\newcommand{\W}[2]{|_{#1}W_{#2}}
\newcommand{\w}[2]{_{#1}W_{#2}}

\newcommand{\Sbar}{\overline{S}}
\newcommand{\Mbar}{\overline{M}}
\newcommand{\Pbar}{\overline{P}\vphantom{P}}
\renewcommand{\hbar}{\overline{h}}
\newcommand{\gbar}{\overline{g}}

\renewcommand{\tilde}{\widetilde}

\renewcommand{\pmatrix}[4]{\left( \begin{smallmatrix} #1 & #2 \\ #3 & #4 \end{smallmatrix} \right)}

\DeclareMathOperator{\spt}{spt}
\DeclareMathOperator{\sptbar}{\overline{spt1}}
\DeclareMathOperator{\mspt}{M2spt}
\DeclareMathOperator{\m}{m2}

\newcommand{\pbar}{\overline{p}}

\newcommand{\pfrac}[2]{\left(\frac{#1}{#2}\right)}

\date{}
\author{Nickolas Andersen}
\thanks{The author acknowledges support
from National Science Foundation grant DMS 08-38434 ``EMSW21-MCTP: Research
Experience for Graduate Students''}

\title[Hecke-type Congruences]{Hecke-type Congruences for Two Smallest Parts Functions}

\begin{document}

\begin{abstract}
We prove infinitely many congruences modulo $3$, $5$, and powers of $2$ for the overpartition function $\pbar(n)$ and two smallest parts functions: $\sptbar(n)$ for overpartitions and $\mspt(n)$ for partitions without repeated odd parts. These resemble the Hecke-type congruences found by Atkin for the partition function $p(n)$ in 1966 and Garvan for the smallest parts function $\spt(n)$ in 2010. The proofs depend on congruences between the generating functions for $\pbar(n), \sptbar(n),$ and $\mspt(n)$ and eigenforms for the half-integral weight Hecke operator $\t{}{\ell^{2}}$.

\end{abstract}

\maketitle
\thispagestyle{empty}

\begin{section}{Introduction and Statement of Results}
Let $\spt(n)$ denote the number of smallest parts in the partitions of $n$. For example, there are 7 partitions of $5$,
\[
	5, 4+1, 3+2, 3+1+1, 2+2+1, 2+1+1+1, 1+1+1+1+1
\]
so $\spt(5)=14$. The $\spt$ function was introduced by Andrews~\cite{Andrews:spt} and has since been studied extensively (see, for instance,~\cite{Ahlgren:ladic, FolsomOno:spt, Garvan:spt-rank, Garvan:spt, Garvan:spt-powers, Ono:spt}). In particular, Garvan gave the following congruences for the spt function in~\cite{Garvan:spt}.

\begin{thm}\label{thm:garvan} Let $s_{\ell} = (\ell^{2}-1)/24$. Then the following are true:
\begin{enumerate}[label=\textup{(\roman*)}]
\item If $\ell>5$ is prime then for any $n\geq 1$, we have
\begin{equation} \label{eq:garvan-mod-72}
	\spt(\ell^{2}n-s_{\ell}) + \pfrac{12}{\ell}\pfrac{1-24n}{\ell}\spt(n)+\ell \spt\pfrac{n+s_{\ell}}{\ell^{2}} \equiv \pfrac{12}{\ell} (1+\ell) \spt (n) \pmod{72}.
\end{equation}
\item If $\ell\geq 5$ is prime, $t=5,7$ or $13$ and $\ell\neq t$ then for any $n\geq 1$, we have
\begin{equation} \label{eq:garvan-mod-t}
	\spt(\ell^{2}n-s_{\ell}) + \pfrac{12}{\ell}\pfrac{1-24n}{\ell}\spt(n)+\ell \spt\pfrac{n+s_{\ell}}{\ell^{2}} \equiv \pfrac{12}{\ell} (1+\ell) \spt (n) \pmod{t}.
\end{equation}
\end{enumerate}
\end{thm}

In this paper we prove results similar to \eqref{eq:garvan-mod-72} and \eqref{eq:garvan-mod-t} for the overpartition function $\pbar(n)$ and for two other smallest parts functions, $\sptbar(n)$ and $\mspt(n)$.

An overpartition is a partition in which the first occurrence of each distinct part may be overlined or not. For example, there are 8 overpartitions of 3,
\[
	3,\, \bar{3},\, 2+1,\, \bar{2}+1,\, 2+\bar{1},\, \bar{2}+\bar{1},\, 1+1+1,\, \bar{1}+1+1,
\]
so $\pbar(3)=8$. The generating function for overpartitions is
\[
	\Pbar(\tau) = \sum_{n=0}^{\infty}\pbar(n)q^{n} := \frac{\eta(2\tau)}{\eta^{2}(\tau)} = 1+2 q+4 q^2+8 q^3+14 q^4+24 q^5+O\left(q^6\right),
\]
where $\eta(\tau)$ is the Dedekind $\eta$-function $\eta(\tau) := q^{1/24}\prod_{n=1}^{\infty}(1-q^{n})$ and $q:=\exp(2\pi i \tau)$.
Define $\sptbar(n)$ to be the number of smallest parts in the overpartitions of $n$ having odd smallest part. For example, the overpartitions of $4$ having odd smallest part are
\[
	3+1,\, \bar{3}+1,\, 3+\bar{1},\, \bar{3}+\bar{1},\, 2+1+1,\, \bar{2}+1+1,\, 2+\bar{1}+1,\, \bar{2}+\bar{1}+1,\, 1+1+1+1,\, \bar{1}+1+1+1,\,
\]
so $\sptbar(4)=20$. The function $\sptbar(n)$ has been studied by many authors including Bringmann, Lovejoy and Osburn~\cite{BringmannLovejoyOsburn}, who proved the congurence
\begin{gather}
	\sptbar(\ell^{2}n) + \pfrac{-n}{\ell}\sptbar(n)+\ell \sptbar(n/\ell^{2}) \equiv  (1+\ell) \sptbar (n) \pmod{3}. \label{eq:spt1bar-mod-3}
\end{gather}
In Theorem \ref{thm:spt1bar} we recover \eqref{eq:spt1bar-mod-3} and extend the modulus to $2^{8} \cdot 3 \cdot 5$.

By~\cite[Section~7]{BringmannLovejoyOsburn}, the generating function for $\sptbar(n)$ is
\begin{gather*}
	\Sbar(\tau) = \sum_{n=1}^{\infty}\sptbar(n)q^{n} = \left(\prod_{n=1}^{\infty}\frac{1+q^{n}}{1-q^{n}}\right)\left(\sum_{n=1}^{\infty}\frac{2nq^{n}}{1-q^{2n}}+\sum_{n\neq 0}\frac{4(-1)^{n}q^{n^{2}+n}(1+q^{2n}+q^{3n})}{(1-q^{2n})(1-q^{4n})}\right) \\
	=2 q+4 q^2+12 q^3+20 q^4+40 q^5+O\left(q^6\right),
\end{gather*}
and we define the function $\Mbar$ by
\begin{gather} \label{eq:mbar-def}
	\Mbar(\tau) := \Sbar(\tau) + \frac{1}{12}\Pbar(\tau)(E_{2}(\tau) - 4E_{2}(2\tau)),
\end{gather}
where $E_{2}$ is the weight 2 quasimodular Eisenstein series
\[
	E_{2}(\tau) := 1-24\sum_{n=1}^{\infty} \sigma_{1}(n) q^{n}.
\]
The function $\Mbar$ is a weight $3/2$ mock modular form. Ahlgren, Bringmann, and Lovejoy~\cite{Ahlgren:ladic} showed that $\Mbar(\tau)$ is an eigenform for the weight $3/2$ Hecke operator $\t{\frac{3}{2}}{\ell^{2}}$, and we use this fact to prove the following theorem.
\begin{thm} \label{thm:spt1bar} Let $\ell$ be an odd prime, and define
\[
	\alpha := \begin{cases} 
					6 & \text{if }\ell \equiv 3\pmod {8}, \\
					7 & \text{if }\ell \equiv 5,7\pmod{8}, \\
					8 & \text{if }\ell \equiv 1 \pmod{8}. \end{cases}
\]
Then for $t\in \{2^{\alpha},3,5\}$, $\ell \neq t$, and $n \geq 1$, we have
\begin{gather}
	\sptbar(\ell^{2}n) + \pfrac{-n}{\ell}\sptbar(n)+\ell \sptbar(n/\ell^{2}) \equiv  (1+\ell) \sptbar (n) \pmod{t}. \label{eq:spt1bar-mod-2}
\end{gather}
\end{thm}

Next, define $\mspt(n)$ to be the restriction of $\spt(n)$ to those partitions without repeated odd parts and whose smallest part is even. For example, $\mspt(7)=3$ since the partitions of $7$ without repeated odd parts are (with even smallest parts circled)
\[
	7,\, 6+1,\, 5+\circled{2},\, 4+3,\, 4+2+1,\, 3+\circled{2}+\circled{2},\, 2+2+2+1.
\]
By~\cite[Section~7]{BringmannLovejoyOsburn}, the generating function for $\mspt(n)$ is
\begin{gather*}
	\sum_{n=1}^{\infty}\mspt(n)q^{n}=\left(\prod_{n=1}^{\infty}\frac{1+q^{2n-1}}{1-q^{2n}}\right)\left(\sum_{n=1}^{\infty}\frac{nq^{2n}}{1-q^{2n}}+\sum_{n\neq 0}\frac{(-1)^{n}q^{2n^{2}+n}}{(1-q^{2n})^{2}}\right) \\
	= q^2+3 q^4+q^5+5 q^6+3 q^7+O\left(q^8\right),
\end{gather*}
and we define functions $S2$ and $M2$ by
\begin{gather}
	S2(\tau) := \sum_{n=0}^{\infty} (-1)^{n} \mspt(n) q^{8n-1}, \\
	M2(\tau) := S2(\tau) + \frac{1}{24}R(\tau)(E_{2}(16\tau) - E_{2}(8\tau)), \label{eq:m2-def} 
\end{gather}
where $R(\tau)$ is the generating function for partitions without repeated odd parts, given by
\[
	R(\tau) = \sum_{n=0}^{\infty} \m(n) q^{8n-1} := \frac{\eta(8\tau)}{\eta^{2}(16\tau)} = \frac{1}{q}-q^7+q^{15}-2 q^{23}+3 q^{31}+O\left(q^{32}\right).
\]
The situation here is  similar to that of $\sptbar(n)$. The function $M2$ is also a weight $3/2$ mock modular form, and the generating functions $\Pbar(\tau)$ and $R(\tau)$ are related under the Fricke involution $\w{}{16}$.
Ahlgren, Bringmann, and Lovejoy~\cite{Ahlgren:ladic} proved that $M2(\tau)$ is an eigenform for the weight 3/2 Hecke operator $\t{\frac{3}{2}}{\ell^{2}}$, and we use this to prove the following result.

\begin{thm} \label{thm:m2spt} Let $\ell$ be an odd prime, and define
\[
	s_{\ell}:=(\ell^{2}-1)/8 \hspace{.1in} \text{ and } \hspace{.1in} 
	\beta := \begin{cases} 
					1 & \text{if } \ell \equiv 3\pmod{8}, \\
					2 & \text{if } \ell \equiv 5\pmod{8}, \\
					3 & \text{if } \ell \equiv 1,7 \pmod{8}. \end{cases}
\]
Then for $t\in\{2^{\beta},3,5\}$, $\ell \neq t$ and $n\geq 1$, we have
\begin{gather} \label{eq:m2spt-mod-t}
	\mspt(\ell^{2}n-s_{\ell}) + \pfrac{2}{\ell}\pfrac{1-8n}{\ell}\mspt(n)+\ell \mspt\pfrac{n+s_{\ell}}{\ell^{2}} \equiv  \pfrac{2}{\ell} (1+\ell) \mspt (n) \pmod{t}.
\end{gather}
\end{thm}

Lastly, we give Hecke-type congruences for overpartitions modulo powers of 2. Many authors have found congruences for $\pbar(n)$ modulo powers of 2 (see, for example~\cite{CorteelLovejoy, HirschhornSellers, Mahlburg:Overpartitions, Kim:overpartitions}), such as
\[
	\pbar(8n+7) \equiv 0 \pmod {64}.
\]
In this paper, we employ the fact that $\Pbar(\tau)$ and $q\frac{d}{dq}\Pbar(\tau)$ are congruent to eigenforms for the Hecke operators $\t{}{\ell^{2}}$ modulo powers of 2 to prove the following theorem.

\begin{thm} \label{thm:pbar} Let $\ell$ be an odd prime, and define
\[
	\gamma := \begin{cases}
					 5 & \text{if } \ell\equiv 3 \pmod 8, \\
					 6 & \text{if } \ell\equiv 5,7\pmod 8, \\
					 7 & \text{if } \ell\equiv 1 \pmod 8.\end{cases}
\]
Then for $n\geq 0$ we have
\begin{gather}
	\pbar(\ell^{2}n) + \pfrac{-n}{\ell} \ell^{-2} \pbar(n) + \ell^{-3} \pbar(n/\ell^{2}) \equiv (1+\ell) \pbar(n) \pmod{16}, \\
	\ell^{2}n\pbar(\ell^{2}n) + \pfrac{-n}{\ell} n \pbar(n) + \ell^{-1} n \pbar(n/\ell^{2}) \equiv (1+\ell) \, n \pbar(n) \pmod{2^{\gamma}}. \label{eq:theta-pbar-mod-2}
\end{gather}
\end{thm}
\end{section}

\begin{section}{Preliminaries}
Let $\lambda$ be a nonnegative integer, $N$ be a positive integer, and $\chi$ be a Dirichlet character modulo $4N$. A holomorphic function $f(\tau)$ on the complex upper half-plane $\mathbb{H}$ is called a (weakly) holomorphic half-integral weight modular form with weight $\lambda + 1/2$ and character $\chi$ if it is holomorphic (resp. meromorphic) at the cusps and if
\[
	f\pfrac{a\tau+b}{c\tau+d} = \chi(d)\pfrac{c}{d}^{2\lambda+1}\epsilon_{d}^{-1-2\lambda}(c\tau+d)^{\lambda+1/2} f(\tau) \text{ for all } \pmatrix{a}{b}{c}{d} \in \Gamma_{0}(4N),
\]
where
\[
	\epsilon_{d} = \begin{cases}
		1 & \text{ if } d\equiv 1\pmod{4}, \\
		i & \text{ if } d\equiv 3\pmod{4}.
	\end{cases}
\]

Suppose that $f$ has the Fourier expansion $f(\tau)=\sum_{n}a(n)q^{n}$. For a prime $\ell$ and a Dirichlet character $\chi$, define the weight $\lambda + 1/2$ Hecke operator $\tt{\lambda+\frac{1}{2}}{\ell^{2}}{\chi}$ (or simply $\t{\lambda+\frac{1}{2}}{\ell^{2}}$ when the character is trivial) by
\begin{equation} \label{eq:hecke-def}
	f \TT{\lambda+\frac{1}{2}}{\ell^{2}}{\chi} (\tau) = \sum_{n}\left( a(\ell^{2}n) + \chi^{*}(\ell)\pfrac{n}{\ell}\ell^{\lambda-1} a(n) + \chi^{*}(\ell^{2}) \ell^{2\lambda-1} a(n/\ell^{2}) \right)q^{n},
\end{equation}
where $\chi^{*}$ is the Dirichlet character given by $\chi^{*}(m):=\left(\frac{(-1)^{\lambda}}{m}\right) \chi(m)$ and $a(n/\ell^{2})=0$ if $\ell^{2} \nmid n$.

A computation shows that the Hecke operator $\t{}{\ell^{2}}$ commutes with the operator $\theta:=q\frac{d}{dq}$ in the following way:

\begin{equation}
	\ell^{2} \theta \left( f \TT{\lambda+\frac{1}{2}}{\ell^{2}}{\chi} \right) = (\theta f) \TT{\lambda+2+\frac{1}{2}}{\ell^{2}}{\chi}. \label{eq:theta-hecke}
\end{equation}

We define the Fricke involution $\w{}{N}$ on $M_{k}^{!}(\Gamma_{0}(N))$ by
\begin{gather} \label{eq:fricke-def}
	F\W{k}{N}(\tau) := \left(-i\sqrt{N}\,\tau\right)^{-k} F(-1/N\tau) .
\end{gather}
Using the fact that $\eta(-1/\tau) = \sqrt{-i\tau} \, \eta(\tau)$, we find that the functions $R(\tau)$ and $\Pbar(\tau)$ are related by
\begin{gather} \label{eq:r-fricke-p}
	R\W{-\frac{1}{2}}{16}(\tau) = \sqrt{-4i\tau}\pfrac{\eta(-1/2\tau)}{\eta^{2}(-1/\tau)} = \sqrt{8} \, \Pbar(\tau).
\end{gather}

We let $M_{k}(\Gamma)$ (resp. $M_{k}^{!}(\Gamma)$) denote the space of modular forms (resp. weakly holomorphic modular forms) of weight $k$ for a congruence subgroup $\Gamma \subset \SL$. For $4\leq k\in 2\Z$, $E_{k}(\tau)$ denotes the weight $k$ Eisenstein series 
\[
	E_{k}(\tau) := 1 - \frac{2k}{B_{k}}\sum_{n=1}^{\infty}\sigma_{k-1}(n) q^{n} \in M_{k}(\SL),
\]
where $B_{k}$ is the $k^{\text{th}}$ Bernoulli number and $\sigma_{m}(n)$ is the sum of divisors function $\sigma_{m}(n) := \sum_{d|n} d^{m}$. Some other useful modular forms are
\begin{gather}
	E(\tau) := 2E_{2}(2\tau) - E_{2}(\tau) \in M_{2}(\Gamma_{0}(2)), \\
	\psi(\tau) := \frac{\eta^{16}(2\tau)}{\eta^{8}(\tau)} = R^{-8}(\tau/8) \in M_{4}(\Gamma_{0}(2)), \\
	\rho(\tau) := \frac{\eta^{16}(\tau)}{\eta^{8}(2\tau)} = 2^{6} \, \psi \W{4}{2}(\tau) = \Pbar^{-8}(\tau) \in M_{4}(\Gamma_{0}(2)). \label{eq:rho-fricke-psi}
\end{gather}
\end{section}

\begin{section}{Proof of Theorem \ref{thm:pbar}}

Let $\ell$ be an odd prime. To prove the statement 
\[
	\pbar(\ell^{2}n) + \pfrac{-n}{\ell} \ell^{-2} \, \pbar(n) + \ell^{-3} \, \pbar(n/\ell^{2}) \equiv (1+\ell) \, \pbar(n) \pmod{16}
\]
from Theorem \ref{thm:pbar}, we will show that
\begin{equation} \label{eq:pbar-hecke-mod-16}
	\Pbar \T{-\frac{1}{2}}{\ell^{2}} (\tau) - (1+\ell)\Pbar(\tau) \equiv 0 \pmod {16}.
\end{equation}
To do this we need the congruence
\[
	\Pbar(\tau) \equiv \Pbar^{-7}(\tau) \pmod {16},
\]
which follows immediately from the fact that $\rho(\tau) \equiv 1\pmod {16}$ and $\rho(\tau) = \Pbar^{-8}(\tau)$. We will also need the following proposition.

\begin{prop} \label{prop:pbar-7-eigenform} Let $\ell$ be an odd prime. Then $\Pbar^{-7}(\tau)$ is an eigenform for the weight $7/2$ Hecke operator $\t{\frac{7}{2}}{\ell^{2}}$ with eigenvalue $\ell^{5}+1$.
\end{prop}

Let us assume for the moment that this proposition is true. 
Since $\ell^{-1}\equiv \ell^{3}\pmod{16}$, the Hecke operators $\t{-\frac{1}{2}}{\ell^{2}}$ and $\t{\frac{7}{2}}{\ell^{2}}$ agree modulo 16, so we have
\begin{align*}
	\Pbar \T{-\frac{1}{2}}{\ell^{2}} - (1+\ell)\Pbar &\equiv \Pbar^{-7} \T{\frac{7}{2}}{\ell^{2}} - (1+\ell)\Pbar^{-7}  \\
	&\equiv (\ell^{5}+1) \Pbar^{-7} - (1+\ell) \Pbar^{-7}  \\
	&\equiv (\ell^{5}-\ell)\Pbar^{-7} \equiv 0 \pmod{16},
\end{align*}
and the first statement in Theorem \ref{thm:pbar} follows.

\begin{proof}[Proof of Proposition \ref{prop:pbar-7-eigenform}]
Recall that the Fricke involution $\w{}{N}$ commutes with the Hecke operator $\t{}{\ell^{2}}$ when $\ell \nmid N$ (see, for example, chapter 1 of~\cite{Ono:web}). Therefore $\Pbar^{-7}(\tau)$ is an eigenform for $\t{}{\ell^{2}}$ if and only if $R^{-7}(\tau)$ is an eigenform for $\t{}{\ell^{2}}$. Since $R^{-7}(\tau) \in M_{\frac{7}{2}}(\Gamma_{0}(16))$, we use the dimension formulas of Cohen-Oesterle (see~\cite{CohenOesterle}) to compute $\dim\left( M_{\frac{7}{2}}(\Gamma_{0}(16)) \right)=8$, and we explicitly compute that there exists a basis of the form $\{f_{m}(\tau)\}_{m=0}^{7}$ where $f_{m}(\tau) = q^{m} + O(q^{8})$. Notice that $R^{-7}(\tau)$ vanishes to order $7$, so $R^{-7}(\tau) = f_{7}(\tau)$. The form $R^{-7}(\tau) = q^{7} + O(q^{15})$ is supported on exponents which are $ 7 \pmod 8$  so the only coefficients which appear in the $q$-expansion of $R^{-7}\T{\frac{7}{2}}{\ell^{2}}$ are those for which the exponent is $7\pmod {8}$. So we have
	\begin{gather*} 
		R^{-7}\T{\frac{7}{2}}{\ell^{2}} = \lambda_{\ell} \, q^{7} + O(q^{15}) \in M_{\frac{7}{2}}(\Gamma_{0}(16))
	\end{gather*}
for some $\lambda_{\ell} \in \Z$ and therefore $R^{-7}\T{\frac{7}{2}}{\ell^{2}} = \lambda_{\ell}  f_{7} = \lambda_{\ell} \, R^{-7}$. To determine the eigenvalue $\lambda_{\ell}$ we use \eqref{eq:r-fricke-p} to compute
 \begin{align*}
 	\Pbar^{-7}\T{\frac{7}{2}}{\ell^{2}} &= 8^{7/2}R^{-7}\W{\frac{7}{2}}{16}\T{\frac{7}{2}}{\ell^{2}}  \\
	&= 8^{7/2}\lambda_{\ell}R^{-7}\W{\frac{7}{2}}{16}  
	= \lambda_{\ell}\Pbar^{-7},
 \end{align*}
so $\lambda_{\ell}$ is the constant term of $\Pbar^{-7}\T{\frac{7}{2}}{\ell^{2}}$ which, by \eqref{eq:hecke-def}, is equal to $\ell^{5}+1$.
\end{proof}

To prove the statement
\[
	\ell^{2}n\pbar(\ell^{2}n) + \pfrac{-n}{\ell} n \pbar(n) + \ell^{-1} n \pbar(n/\ell^{2}) \equiv (1+\ell) \, n \pbar(n) \pmod{2^{\gamma}}
\]
from Theorem \ref{thm:pbar}, we will prove the equivalent statement
\begin{gather} \label{eq:theta-pbar-hecke-mod-64}
	\left(\theta \Pbar\right) \T{\frac{3}{2}}{\ell^{2}}(\tau) - (1+\ell)\theta \Pbar(\tau) \equiv 0 \pmod{2^{\gamma}},
\end{gather}
where we recall that
\[
	\gamma=\begin{cases}
					 5 & \text{if } \ell\equiv 3 \pmod 8 \\
					 6 & \text{if } \ell\equiv 5,7\pmod 8 \\
					 7 & \text{if } \ell\equiv 1 \pmod 8.\end{cases}
\]
To do this we will proceed as before, replacing $\theta \Pbar(\tau)$ by $\theta\left(\Pbar^{-31}(\tau) + 64\Pbar(\tau)\psi(\tau)\right)$, a switch that is justified by the following lemma.

\begin{lem}
\begin{gather} \label{eq:theta-p-mod-64}
	\theta\Pbar \equiv \theta\left(\Pbar^{-31} + 64\Pbar\psi\right) \pmod{128}.
\end{gather}
\end{lem}

\begin{proof}
Since $\rho(\tau) = \Pbar^{-8}(\tau)$, an equivalent form of \eqref{eq:theta-p-mod-64} is
	\begin{gather} \label{eq:theta-p-rho-psi}
		\theta(\Pbar(1-\rho^{4}-64\psi)) \equiv 0 \pmod{128}.
	\end{gather}
To prove \eqref{eq:theta-p-rho-psi} we will need the derivatives
\begin{gather}
	\theta\rho = -\frac{1}{3}\rho(E - E_{2}), \label{eq:rho-deriv} \\
	\theta\psi = \frac{1}{3}\psi(2E + E_{2}). \label{eq:psi-deriv}
\end{gather}
Since $\rho^{4}\equiv 1\pmod{64}$ and $\Pbar\equiv 1\pmod{2}$ we have
\begin{align*}
	\theta(\Pbar(1-\rho^{4}-64\psi)) &= (\theta \Pbar)(1-\rho^{4}-64\psi) + \Pbar \, \theta(1-\rho^{4}-64\psi) \\
	&\equiv \Pbar \, \left( \frac{4}{3}\rho^{4}(E-E_{2}) - \frac{64}{3} \psi (2E+E_{2}) \right) \pmod{128}.
\end{align*}
Since $2E+E_{2} \equiv 1\pmod{2}$, it remains to show that 
\begin{equation}
	E-E_{2}\equiv 16\psi \pmod{32}. \label{eq:e-e2-16-psi-mod-32}
\end{equation}
Computation shows that
\[
	16\psi(\tau) = \frac{16}{240} \left( E_{4}(\tau) - E_{4}(2\tau) \right) = 16\sum_{n=1}^{\infty}\left( \sigma_{3}(n) - \sigma_{3}(n/2) \right) q^{n}.
\]	
We also have
\[
	E(\tau) - E_{2}(\tau) = 2(E_{2}(2\tau) - E_{2}(\tau)) = 48\sum_{n=1}^{\infty}\left( \sigma_{1}(n) - \sigma_{1}(n/2) \right)q^{n}.
\]
Congruence \eqref{eq:e-e2-16-psi-mod-32} follws since $\sigma_{3}(n) \equiv \sigma_{1}(n) \pmod{2}$.
\end{proof}

We return now to the proof of \eqref{eq:theta-pbar-hecke-mod-64}; in light of \eqref{eq:theta-hecke} and \eqref{eq:theta-p-mod-64}, it is enough to prove the statements
\begin{gather}
	\theta\left( \ell^{2} \Pbar^{-31}\T{-\frac{1}{2}}{\ell^{2}} - (1+\ell)\Pbar^{-31} \right)\equiv 0 \pmod{2^{\gamma}}, \label{eq:theta-pbar-hecke-weight-1/2} \\
	\theta \left(\ell^{2} \Pbar\psi \T{-\frac{1}{2}}{\ell^{2}} - (1+\ell)\Pbar\psi \right)\equiv 0 \pmod 2. \label{eq:theta-pbar-psi-mod-2}
\end{gather}
We would like to exchange weight $-1/2$ for weight $31/2$ in the first case and weight $7/2$ in the second case. The exchange is trivial in \eqref{eq:theta-pbar-psi-mod-2}. To see that it is justified in \eqref{eq:theta-pbar-hecke-weight-1/2}, let $\Pbar^{-31}(\tau) = \sum_{n=0}^{\infty} a(n) q^{n}$. Since $\Pbar(\tau) \equiv 1\pmod 2$, $a(n)$ is even for all $n\geq 1$. Hence
\[
	\ell^{2}\Pbar^{-31}\T{\frac{31}{2}}{\ell^{2}}(\tau) - \ell^{2}\Pbar^{-31}\T{-\frac{1}{2}}{\ell^{2}}(\tau) \equiv \sum_{n=1}^{\infty} \pfrac{-n}{\ell} (\ell^{16}-1) a(n) q^{n} \equiv 0 \pmod{128},
\]
since $\ell^{16}\equiv 1\pmod {64}$. Therefore \eqref{eq:theta-pbar-hecke-weight-1/2} is implied by
\begin{gather}
	\theta\left( \ell^{2} \Pbar^{-31}\T{\frac{31}{2}}{\ell^{2}}(\tau) - (1+\ell)\Pbar^{-31}(\tau) \right)\equiv 0 \pmod{2^{\gamma}}. \label{eq:theta-pbar-hecke-weight-15/2}
\end{gather}
The truth of \eqref{eq:theta-pbar-hecke-weight-15/2} will follow from the next proposition.

\begin{prop} \label{prop:pbar-31-eigenform} The function $\Pbar^{-31}(\tau)$ is an eigenform for the Hecke operator $\t{\frac{31}{2}}{\ell^{2}}$ modulo $2^{12}$ with eigenvalue $\ell^{29}+1$.
\end{prop}

Assuming for the moment that this proposition is true, we have
\begin{align*}
	\ell^{2} \Pbar^{-31}\T{\frac{31}{2}}{\ell^{2}}(\tau) - (1+\ell)\Pbar^{-31}(\tau) &\equiv (\ell^{2}(\ell^{29}+1) - (1+\ell))\Pbar^{-31}(\tau) \pmod{2^{12}}.
	\end{align*}
Then \eqref{eq:theta-pbar-hecke-weight-15/2} follows from the congruence
	\[
		\ell^{31}+\ell^{2}-\ell-1 \equiv 0\pmod {2^{\gamma-1}},
	\] 
which is easily checked.

\begin{proof}[Proof of Proposition \ref{prop:pbar-31-eigenform}]
Since the form $R^{-1}(\tau)$ is supported on exponents which are $7\pmod {8}$ and the form $R(\tau)$ is supported on exponents which are $1\pmod {8}$, we conclude that
\[
	\tilde{R}_{\ell}(\tau) := R^{-31}\T{\frac{31}{2}}{\ell^{2}}(\tau) \cdot R^{7}(\tau) \in M_{12}^{!}(\Gamma_{0}(16))
\]
is supported on exponents divisible by 8. By \cite[Lemma~7]{AtkinLehner}, we have $\tilde{R}_{\ell}(\tau/8) \in M_{12}^{!}(\Gamma_{0}(2))$, and a computation involving the Fricke involution $\w{}{2}$ shows that $\tilde{R}_{\ell}(\tau/8)$ is holomorphic at the cusp 0. So $\tilde{R}_{\ell}(\tau/8) \in M_{12}(\Gamma_{0}(2))$, which implies that there exist integers $c_{j}$, $0\leq j\leq 3$, such that
\[
	\tilde{R}_{\ell}(\tau/8) = c_{0}\rho^{3}(\tau) + c_{1}\rho^{2}(\tau)\psi(\tau) + c_{2}\rho(\tau)\psi^{2}(\tau) + c_{3}\psi^{3}(\tau).
\]
Since $\psi(8\tau) = R^{-8}(\tau)$, we have
\begin{gather} \label{eq:r-15-hecke-r-7-rho-r-15}
	R^{-31}\T{\frac{31}{2}}{\ell^{2}}(\tau) = c_{0}R^{-7}(\tau)\rho^{3}(8\tau) + c_{1}R^{-15}(\tau)\rho^{2}(8\tau) + c_{2}R^{-23}(\tau)\rho(8\tau) + c_{3} R^{-31}(\tau).
\end{gather}
Applying the Fricke involution $\w{}{16}$ to \eqref{eq:r-15-hecke-r-7-rho-r-15} and multiplying by $8^{31/2}$, we obtain
\[
	\Pbar^{-31}\T{\frac{31}{2}}{\ell^{2}} = 2^{36} \, c_{0} \, \Pbar^{-7}\psi^{3} + 2^{24} \, c_{1} \, \Pbar^{-15}\psi^{2} + 2^{12} \, c_{2} \, \Pbar^{-23}\psi + c_{3} \,\Pbar^{-31}(\tau).
\]
Using the fact that $\psi$ vanishes at $\infty$ and recalling \eqref{eq:hecke-def}, we conclude that $c_{3}=\ell^{29}+1$. Hence,
\[
	\Pbar^{-31}\T{\frac{31}{2}}{\ell^{2}} \equiv (\ell^{29}+1)\Pbar^{-31}(\tau) \pmod{2^{12}},
\]
which completes the proof of Proposition \ref{prop:pbar-31-eigenform}.
\end{proof}

It remains only to prove \eqref{eq:theta-pbar-psi-mod-2} to finish the proof of Theorem \ref{thm:pbar}.
Since 
$\psi(\tau) \equiv \sum_{n=0}^{\infty}q^{(2n+1)^{2}} \pmod{2}$, we only need to show that 
\[
	\left( \sum_{n=0}^{\infty}q^{(2n+1)^{2}} \right)\T{\frac{7}{2}}{\ell^{2}} \equiv 0 \pmod 2.
\]
Let $\sum_{n=0}^{\infty}q^{(2n+1)^{2}} = \sum_{n=1}^{\infty}b(n) q^{n}$. Then
\[
	b(\ell^{2}n) + \pfrac{-n}{\ell}b(n) + b(n/\ell^{2}) \equiv \begin{cases}
	b(\ell^{2}n) + b(n) \equiv 1+1 \equiv 0\pmod 2 & \text{if } n=m^{2}, m \text{ odd}, (m,\ell)=1 \\
	b(\ell^{2}n) + b(m^{2}) \equiv 1+1 \equiv 0 \pmod{2} & \text{if } n=\ell^{2}m^{2}, m \text{ odd} \\
	b(\ell^{2}n) + \pfrac{-n}{\ell} b(n) \equiv 0 + 0 \equiv 0 \pmod{2} & \text{if } n\neq \text{odd square}.
	 \end{cases}
\]
Congruence \eqref{eq:theta-pbar-psi-mod-2} follows, and Theorem \ref{thm:pbar} is proved.\qed

\end{section}

\begin{section}{Proof of theorems \ref{thm:spt1bar} and \ref{thm:m2spt} modulo 3 and 5}
In this section we prove the congruences \eqref{eq:spt1bar-mod-2} and \eqref{eq:m2spt-mod-t} for the moduli 3 and 5. The proofs use similar techniques, so we present them together. 
Let $\ell$ be an odd prime and define $t_{\ell} := 15/(15,\ell)$.

Recalling that $\Sbar(\tau) = \sum_{n}\sptbar(n) q^{n}$, we find that Theorem \ref{thm:spt1bar} for the moduli $t=3,5$ is equivalent to
\begin{gather} \label{eq:s-ell-mod-t-ell}
	\Sbar_{\ell}(\tau):= \Sbar\T{\frac{3}{2}}{\ell^{2}}(\tau) - (1+\ell)\Sbar(\tau) \equiv 0 \pmod{t_{\ell}}.
\end{gather}
We will need the following lemma.

\begin{lem} \label{4e2-3e-rho-mod-45} $4E_{2}(2\tau) - E_{2}(\tau) \equiv 3E(\tau)\rho(\tau) \pmod{45}$.
\end{lem}

\begin{proof}
Since $E(\tau)\rho(\tau) \in M_{6}(\Gamma_{0}(2))$, it is a linear combination of $E_{6}(\tau)$ and $E_{6}(2\tau)$. We compute that
\begin{gather*}
	3E(\tau)\rho(\tau) = \frac{1}{21}\left( 64E_{6}(2\tau) - E_{6}(\tau) \right) = 
	3- 24\sum_{n=1}^{\infty}\left( 64\sigma_{5}(n/2) - \sigma_{5}(n) \right) q^{n} \\
	\equiv 4 E_{2}(2\tau) - E_{2}(\tau) \pmod{45}. \qedhere
\end{gather*}
\end{proof}

We apply Lemma \ref{4e2-3e-rho-mod-45} to equation \eqref{eq:mbar-def} to obtain
\begin{gather*}
	\Mbar(\tau) - \Sbar(\tau) \equiv -\frac{1}{4}\Pbar(\tau) \cdot E(\tau)\rho(\tau) 
			\equiv 11 \Pbar^{-7}(\tau) E(\tau) \pmod{15}.
\end{gather*}
We omit the proof of the next proposition since it is similar to the proof of Proposition \ref{prop:pbar-7-eigenform}.

\begin{prop} \label{prop:p-7-e-eigenform} Let $\ell$ be an odd prime. Then $E(\tau)\Pbar^{-7}(\tau)$ is an eigenform for $\t{\frac{11}{2}}{\ell^{2}}$ with eigenvalue $\ell^{9}+1$.
\end{prop}

Since $\ell^{4}\equiv 1\pmod{t_{\ell}}$, the Hecke operators $\t{\frac{3}{2}}{\ell^{2}}$ and $\t{\frac{11}{2}}{\ell^{2}}$ agree modulo $t_{\ell}$. The function $\Mbar$ is an eigenform for $\t{\frac{3}{2}}{\ell^{2}}$, therefore
\[
	\Mbar\T{\frac{11}{2}}{\ell^{2}} \equiv (1+\ell)\Mbar \pmod {t_{\ell}},
\]
from which we can conclude
\begin{gather*}
	\Sbar_{\ell} \equiv -\left( 11(\ell^{9}+1)\Pbar^{-7}E - 11(\ell+1)\Pbar^{-7}E \right) 
	\equiv 4\ell (\ell^{8}-1) \Pbar^{-7}(\tau) E(\tau) 
	\equiv 0\pmod{t_{\ell}}.
\end{gather*}
This completes the proof of \eqref{eq:s-ell-mod-t-ell}.

To prove Theorem \ref{thm:m2spt} for the moduli $t=3,5$ we will prove that
\begin{equation} \label{eq:s2-ell-mod-t-ell}
	S2_{\ell}(\tau):= S2\T{\frac{3}{2}}{\ell^{2}}(\tau) - (1+\ell)S2(\tau) \equiv 0 \pmod{t_{\ell}},
\end{equation}
where we recall that $S2(\tau) = \sum_{n}(-1)^{n}\mspt(n) q^{8n-1}$. We omit the proof of the next lemma since it is similar to the proof of Lemma \ref{4e2-3e-rho-mod-45}.

\begin{lem} \label{lem:e2-e-psi-mod-45} 
$E_{2}(2\tau) - E_{2}(\tau) \equiv 24\,E(\tau)\psi(\tau) \pmod{45}$.
\end{lem}

We apply Lemma \ref{lem:e2-e-psi-mod-45} to \eqref{eq:m2-def} to obtain
\begin{gather*}
M2(\tau) - S2(\tau) \equiv R(\tau) \cdot E(8\tau)\psi(8\tau)
			\equiv R(\tau)^{-7} E(8\tau) \pmod{15}.
\end{gather*}
Since $R(\tau)^{-7}E(8\tau) \W{\frac{11}{2}}{16} = 2^{-27/2}E(\tau)\Pbar^{-7}(\tau)$, and since $\w{}{16}$ commutes with $\t{}{\ell^{2}}$ we have, by Proposition \ref{prop:p-7-e-eigenform}, that $E(8\tau)R(\tau)^{-7}$ is an eigenform for $\t{\frac{11}{2}}{\ell^{2}}$ with eigenvalue $\ell^{9}+1$. As in the case of $\Mbar$, we are justified in switching the weight from $\frac{3}{2}$ to $\frac{11}{2}$, and since $M2$ is an eigenform for $\t{\frac{3}{2}}{\ell^{2}}$, we have
\[
	M2\T{\frac{11}{2}}{\ell^{2}} \equiv (1+\ell)\Mbar \pmod {t_{\ell}},
\]
from which we can conclude
\begin{gather*}
	S2_{\ell}(\tau) \equiv -\left[ (\ell^{9}+1)R(\tau)^{-7}E(8\tau) - (\ell+1)R(\tau)^{-7}E(8\tau) \right] 
	\equiv -\ell (\ell^{8}-1) R(\tau)^{-7}E(8\tau)
	\equiv 0\pmod{t_{\ell}}.
\end{gather*}
This completes the proof of \eqref{eq:s-ell-mod-t-ell}. \qed
\end{section}

\begin{section}{Proof of Theorem \ref{thm:spt1bar} modulo powers of 2} \label{sec:spt1bar}
Recalling the definition of $\Sbar_{\ell}$ in \eqref{eq:s-ell-mod-t-ell}, we must prove that
\[
	\Sbar_{\ell}(\tau) \equiv 0 \pmod{2^{\alpha}},
\]
where
\[
	\alpha = \begin{cases} 
					6 & \text{if }\ell \equiv 3\pmod 8 \\
					7 & \text{if }\ell \equiv 5,7\pmod{8} \\
					8 & \text{if }\ell \equiv 1 \pmod{8}. \end{cases} 
\]
By \eqref{eq:mbar-def} and \eqref{eq:rho-deriv} we obtain
\begin{equation}
	\Mbar(\tau) - \Sbar(\tau) = 2 \theta \Pbar(\tau) - \frac{1}{4}\hbar(\tau),
\end{equation}
where $\hbar(\tau)$ is the weight 3/2 weakly holomorphic modular form
\[
	\hbar(\tau) := E(\tau)\frac{\eta(2\tau)}{\eta^{2}(\tau)} \in M_{\frac{3}{2}}^{!}(\Gamma_{0}(16)).
\]
Using \eqref{eq:theta-hecke} and the fact that $\Mbar \T{\frac{3}{2}}{\ell^{2}}(\tau) = (1+\ell)\Mbar(\tau)$ we obtain
\begin{align} \label{eq:spt1bar-ell-h-p}
	\Sbar_{\ell}(\tau)
	= \frac{1}{4}\left( h \T{\frac{3}{2}}{\ell^{2}}(\tau) - (1+\ell)\hbar(\tau) \right) - 2\theta\left( \ell^{2} \Pbar \T{-\frac{1}{2}}{\ell^{2}}(\tau) - (1+\ell) \Pbar(\tau) \right),
\end{align}
so to prove \eqref{eq:spt1bar-mod-2} it is enough to prove the statements
\begin{gather}
	\hbar \T{\frac{3}{2}}{\ell^{2}}(\tau) - (1+\ell)\hbar(\tau) \equiv 0 \pmod{256}, \\ 
	\theta\left( \ell^{2}\Pbar \T{-\frac{1}{2}}{\ell^{2}} - (1+\ell) \Pbar(\tau) \right) \equiv 0 \pmod{2^{\alpha-1}}.
\end{gather}
The second statement has been established in \eqref{eq:theta-pbar-hecke-mod-64}, and the first statement follows from the next proposition, which completes the proof of Theorem \ref{thm:spt1bar}.

\begin{prop}
	\begin{equation} \label{eq:h-ell-mod-4096}
		\hbar \T{\frac{3}{2}}{\ell^{2}}(\tau) - (1+\ell)\hbar(\tau) \equiv 0 \pmod {2^{12}}.
	\end{equation}
\end{prop}

\begin{proof}
Define $\gbar(\tau) \in M_{\frac{3}{2}}^{!}(\Gamma_{0}(16))$ by
\[
	\gbar(\tau) := \frac{1}{\sqrt{8}}\hbar(\tau) \W{\frac{3}{2}}{16} = E(8\tau)\frac{\eta(8\tau)}{\eta^{2}(16\tau)}.
\]
From~\cite[Corollary~4]{AhlgrenKim}, we have
\[
	\gbar \T{\frac{3}{2}}{\ell^{2}} (\tau) - \gbar (\tau) = \ell f_{\ell^{2}}(\tau)
\]
where $f_{\ell^{2}}(\tau)$ is given by
\[
	f_{\ell^{2}}(\tau) = \left( 1+ \sum_{n=1}^{(\ell^{2}-1)/8} c_{n} j_{2}^{n}(8\tau) \right) \gbar(\tau)
\]
for some $c_{n} \in \Z$, and $j_{2}(\tau)$ is the Hauptmodul on $\Gamma_{0}(2)$ given by
\[
	j_{2}(\tau) := \left(\frac{\eta(\tau)}{\eta(2\tau)}\right)^{24} \in M_{0}^{!}(\Gamma_{0}(2)).
\]
Applying the Fricke involution $\w{}{16}$ and multiplying by $\sqrt 8$, we obtain
\[
	\hbar(\tau) \T{\frac{3}{2}}{\ell^{2}} - \hbar(\tau) = \ell\left(1+\sum_{n=1}^{s_{\ell}}c_{n}j_{2}^{n}(-1/2\tau)\right) \hbar(\tau).
\]
Since $j_{2}(-1/2\tau) = 2^{12} \, j_{2}^{-1}(\tau)$, we conclude that
\[
	\hbar(\tau) \T{\frac{3}{2}}{\ell^{2}} - \hbar(\tau) \equiv \ell \, \hbar(\tau) \pmod {2^{12}},
\]
from which \eqref{eq:h-ell-mod-4096} follows.
\end{proof}
\end{section}

\begin{section}{Proof of Theorem \ref{thm:m2spt} modulo powers of 2}
Recalling the definition of $S2_{\ell}$ in \eqref{eq:s2-ell-mod-t-ell}, we must prove that
\begin{gather}
	S2_{\ell}(\tau) \equiv 0\pmod {2^{\beta}},
\end{gather}
where
\[
	\beta := \begin{cases} 
					1 & \text{if } \ell \equiv 3\pmod{8} \\
					2 & \text{if } \ell \equiv 5\pmod{8} \\
					3 & \text{if } \ell \equiv 1,7 \pmod{8}. \end{cases}
\]
We will need the following lemma.
\begin{lem} \label{lem:e2-e-psi-mod-64} 
$E_{2}(2\tau) - E_{2}(\tau) \equiv 24\, \psi(\tau) - 16\, \psi^{2}(\tau) + 32\, \psi^{4}(\tau) \pmod{64}$.
\end{lem}

\begin{proof}
	If $m\in \Z$ then $E_{2}(\tau) - m E_{2}(m\tau) \in M_{2}(\Gamma_{0}(m))$. Therefore
	\[
		g(\tau) := E_{2}(2\tau) - 64E_{2}(128\tau) + 64E_{2}(64\tau) - E_{2}(\tau) \in M_{2}(\Gamma_{0}(128)),
	\]
and $g(\tau) \equiv E_{2}(2\tau) - E_{2}(\tau) \pmod{64}$. We compute that
\[
	g(\tau) \equiv \psi(\tau) - 16\, \psi^{2}(\tau) + 32\, \psi^{4}(\tau) \pmod{64}
\]
by computing sufficiently many terms of 
\[
	\tilde{g}(\tau) := g(\tau)E(\tau)\rho^{3}(\tau) - 24\psi(\tau)\rho^{3}(\tau) + 16\psi^{2}(\tau)\rho^{2}(\tau) - 32\psi^{4}(\tau) \in M_{16}(\Gamma_{0}(128))
\]
to see that $\tilde{g}(\tau) \equiv 0\pmod{64}$. Since $\tilde{g}(\tau) \equiv g(\tau) - 24\psi(\tau) + 16\psi^{2}(\tau) - 32\psi^{4}(\tau) \pmod{64}$, this completes the proof.

\end{proof}

We apply Lemma \ref{lem:e2-e-psi-mod-64} to \eqref{eq:m2-def}
along with the fact that $\psi(8\tau) = R^{-8}(\tau)$ to obtain
\[
	M2 - S2 \equiv R^{-7} - 2R^{-15} + 4R^{-31} \pmod{8}.
\]
Let $\ell$ be an odd prime. By Proposition \ref{prop:pbar-7-eigenform}, we have $R^{-7}\T{\frac{7}{2}}{\ell^{2}} =  (\ell^{5}+1) R^{-7}$.
Therefore
\[
	-S2_{\ell} \equiv 2\left( R^{-15}\T{\frac{15}{2}}{\ell^{2}} - (1+\ell)R^{-15} \right) + 4\left( R^{-31}\T{\frac{31}{2}}{\ell^{2}} - (1+\ell)R^{-31} \right) \pmod{8}.
\]
This proves Theorem \ref{thm:m2spt} when $\ell\equiv 3\pmod{8}$. Suppose now that $\ell\not\equiv 3\pmod{8}$. To simplify notation, let $\rho_{8} = \rho(8\tau)$. Recall \eqref{eq:r-15-hecke-r-7-rho-r-15} and the discussion that follows it, which together imply that for some integers $c_{0},c_{1},c_{2}$ depending on $\ell$, we have
\begin{equation}
	R^{-31}\T{\frac{31}{2}}{\ell^{2}} = c_{0}R^{-7}\rho_{8}^{3} + c_{1}R^{-15}\rho_{8}^{2} + c_{2}R^{-23}\rho_{8} + (\ell^{29}+1) R^{-31}. \label{eq:r-31-basis-exp}
\end{equation}
Similarly, it can be shown that there exists some $d_{0}\in \Z$ depending on $\ell$ such that
\begin{equation}
	R^{-15}\T{\frac{15}{2}}{\ell^{2}}(\tau) = d_{0} R^{-7}\rho_{8} + (\ell^{13}+1) R^{-15}, \label{eq:r-15-basis-exp}
\end{equation}
and since $\rho \equiv 1\pmod{8}$, we have
\begin{gather*}
	-S2_{\ell} \equiv 2\left( d_{0} R^{-7}\rho_{8} + (\ell^{13}-\ell)R^{-15} \right) + 4 \left( c_{0}R^{-7}\rho_{8}^{3} + c_{1}R^{-15}\rho_{8}^{2} + c_{2}R^{-23}\rho_{8} + (\ell^{29}-\ell) R^{-31} \right) \\
	\equiv 2 d_{0} R^{-7} + 4 \left( c_{0}R^{-7} + c_{1}R^{-15} + c_{2}R^{-23} \right) \pmod{8}.
\end{gather*}
The following proposition completes the proof.

\begin{prop} Let $c_{0},c_{1},c_{2},d_{0}$ be as above. Then the following are true:
\begin{enumerate}
	\item If $\ell\equiv 1\pmod{4}$ then $d_{0}\equiv 0\pmod{4}$.
	\item If $\ell\equiv 1\pmod{8}$ then $c_{0}\equiv 0\pmod{2}$.
	\item If $\ell\equiv 7\pmod{8}$ then $d_{0}\equiv 2c_{0} \pmod{4}$.
	\item If $\ell\equiv 1,7 \pmod{8}$ then $c_{1}\equiv c_{2} \equiv 0\pmod{2}$.
\end{enumerate}
\end{prop}

\begin{proof} By \eqref{eq:hecke-def}, \eqref{eq:r-31-basis-exp}, \eqref{eq:r-15-basis-exp}, and the fact that
	$R^{-7}\rho_{8}^{3} = q^{7} - 41q^{15} + 789q^{23} + O(q^{31})$ and 
	$R^{-15}\rho_{8}^{2} = q^{15} - 17q^{23} + O(q^{31})$,
the coefficients $d_{0},c_{0},c_{1},c_{2}$ are given by
\begin{gather*}
	d_{0} = a_{15}(7\ell^{2}), \hspace{.2in} c_{0} = a_{31}(7\ell^{2}), \hspace{.2in} 
	c_{1} = a_{31}(15\ell^{2}) + 41c_{0}, \hspace{.2in} c_{2} = a_{31}(23\ell^{2}) + 17c_{1} - 789c_{0},
\end{gather*}
where 
\[
	\sum_{n=1}^{\infty}a_{r}(n) q^{n} := R^{-r}(\tau) = \left( \sum_{n=0}^{\infty}q^{(2n+1)^{2}} \right)^{r}.
\]

For $f(\tau) = \sum_{n=1}^{\infty} a(n)q^{n} \in M_{k+1/2}(\Gamma_{0}(4N))$ we define the $t^{\text{th}}$ Shimura lift  $\Shim_{t}$ by
\begin{equation} \label{eq:shimura-def}
	\Shim_{t}f(\tau) = \sum_{n=1}^{\infty}\left(\sum_{d|n} d^{k-1}\pfrac{(-1)^{k}4t}{d}a\left(|t|n^{2}/d^{2}\right) \right)q^{n}.
\end{equation}
If $f\in S_{k+1/2}(\Gamma_{0}(4N))$, then $\Shim_{t}f \in S_{2k}(\Gamma_{0}(2N))$ (see, for instance,~\cite[Section~3.3]{Ono:web}). 
Now, $\rho\,\psi$ is a cusp form since $\psi$ vanishes at $\infty$ and $\rho$ vanishes at $0$. 
We compute $\Shim_{t} \rho_{8} R^{-r}$, which is a cusp form.  
Since $\rho \equiv 1\pmod{4}$ we have
\[
	\Shim_{t}\rho_{8}R^{-r} \equiv \Shim_{t}R^{-r} \pmod{4}.
\]
	We will now compute the lifts $\Shim_{t}$ and write the resulting functions in terms of $F$ and $\vartheta_{0}$, where
	\[
		F := \sum_{n=0}^{\infty}\sigma_{1}(2n+1) q^{2n+1} \in M_{2}(\Gamma_{0}(4)), \hspace{.2in} \vartheta_{0} := \sum_{n=-\infty}^{\infty}q^{n^{2}} \in M_{\frac{1}{2}}(\Gamma_{0}(4)).
	\]
	
(1)
Suppose $\ell\equiv 1\pmod 4$. We compute sufficiently many coefficients to determine that
\begin{align*}
	\Shim_{7}\rho_{8}R^{-15} &= -1287 F^{3} \vartheta_{0}^{32} + 2^{6}\cdot 1287 F^{4} \vartheta_{0}^{28} - 2^{8}\cdot 6721 \, F^{5} \vartheta_{0}^{24} + 2^{12} \cdot 2145 \, F^{6} \vartheta_{0}^{20} \\
	&+ 2^{16}  \cdot 1859 \, F^{7}\vartheta_{0}^{16} - 2^{20}\cdot 1287 \, F^{8} \vartheta_{0}^{12} + 2^{24}\cdot 143 \, F^{9} \vartheta_{0}^{8}.
\end{align*}
Since $\vartheta_{0}(\tau) \equiv 1\pmod 2$, we conclude that
\[
	\Shim_{7}R^{-15} \equiv \Shim_{7}\rho_{8}R^{-15} \equiv  F^{3} \pmod{4}.
\]
Now, $a_{15}(7\ell^{2})$ is the coefficient of $q^{\ell}$ in $\Shim_{7}R^{-15}$. We must show that the form $F^{3} \pmod{4}$ is supported on exponents which are $3 \pmod{4}$. Let $\chi_{-4}:=\pfrac{-4}{\bullet}$, and let $F_{\chi_{-4}}$ denote the twist of $F$ by $\chi_{-4}$. Then
\[
	\tilde{F} := \frac{1}{2}\left( F^{3} + F^{3}_{\chi_{-4}} \right) \in M_{2}(\Gamma_{0}(16)),
\]
and by computing sufficiently many coefficients we find that $\tilde{F} \equiv 0 \pmod{4}$.

(2) 
Suppose $\ell\equiv 1\pmod{8}$. We compute that
\begin{align}
		\Shim_{7}\rho_{8}R^{-31} &= \nonumber
	-693\,F^{3}\vartheta_{0}^{64} 
	+2^{7}\cdot 693\,F^{4}\vartheta_{0}^{60}
	-14158837\,F^{5}\vartheta_{0}^{56} 
	+2^{4}\cdot 74274739\,F^{6}\vartheta_{0}^{52} \\
	&-45253573295\,F^{7}\vartheta_{0}^{48} \nonumber
	+2^{5}\cdot 20433347725\,F^{8}\vartheta_{0}^{44} 
	+2^{8}\cdot 29560308687\,F^{9}\vartheta_{0}^{40} \\
	&-2^{14}\cdot 28133747817\,F^{10}\vartheta_{0}^{36} \nonumber
	+2^{15}\cdot 250545162231\,F^{11}\vartheta_{0}^{32} 
	-2^{23}\cdot 9410428671\,F^{12}\vartheta_{0}^{28} \\
	&+2^{23}\cdot 53378995173\,F^{13}\vartheta_{0}^{24} \nonumber
	-2^{27}\cdot 11400290027\,F^{14}\vartheta_{0}^{20} 
	+2^{32}\cdot 697257169\,F^{15}\vartheta_{0}^{16} \\
	&-2^{35}\cdot 43328593\,F^{16} \vartheta_{0}^{12} 
	-2^{40}\cdot 122815\,F^{17}\vartheta_{0}^{8}, \label{eq:shim-r-31}
\end{align}
and therefore
\[
	\Shim_{7}R^{-31} \equiv F^{3} + F^{5} + F^{7} \pmod{2}.
\]
Notice that $F\equiv \sum_{n\geq0}q^{(2n+1)^{2}} \pmod{2}$, so the the form $F^{3} + F^{5} + F^{7} \pmod{2}$ is supported only on exponents which are $3,5,7 \pmod{8}$. The quantity $c_{0} = a_{31}(7\ell^{2})$ is the coefficient of $q^{\ell}$ in $\Shim_{7}R^{-31}$, therefore $c_{0}\equiv 0\pmod{2}$.

(3) 
Suppose $\ell\equiv 7\pmod{8}$. From \eqref{eq:shim-r-31} we have
\[
	\Shim_{7}R^{-31} \equiv 3F^{3} + 3F^{5} + F^{7} \pmod{4}.
\]
The quantity $d_{0} - 2c_{0} = a_{15}(7\ell^{2})-2a_{31}(7\ell^{2})$ is the coefficient of $q^{\ell}$ in 
\[
	\Shim_{7}R^{-15} - 2 \, \Shim_{7}R^{-31} \equiv 3F^{3}+2F^{5}+2F^{7} \pmod{4}.
\]
Since $F\equiv \sum_{n\geq0}q^{(2n+1)^{2}} \pmod{2}$, the form $2F^{5} \pmod{4}$ is supported on exponents which are $5 \pmod{8}$.
We must show that the form $3F^{3}+2F^{7} \pmod{4}$ is supported on exponents which are 1 or $3\pmod{8}$. Then, since $\ell\equiv 7\pmod{8}$, it will follow that $d_{0} - 2c_{0}\equiv 0\pmod{4}$.

Define
\[
	f := 3F^{3}\vartheta_{0}^{16} + 2F^{7} \in M_{14}(\Gamma_{0}(4)).
\]
Then $f \equiv 3F^{3}+2F^{7} \pmod{4}$. Let $\chi_{-8}:=\pfrac{-8}{\bullet}$, and let $f_{\chi_{-8}}$ denote the twist of $f$ by $\chi_{-8}$. Then
\[
	\tilde{f} := \frac{1}{2} \left( f - f_{\chi_{-8}} \right) \in M_{14}(\Gamma_{0}(64)),
\]
and by computing sufficiently many coefficients we find that $\tilde{f} \equiv 0 \pmod{4}$.

(4)
Suppose $\ell\equiv 1,7 \pmod{8}$. To show that $c_{1}$ and $c_{2}$ are even, we will show
\begin{gather}
	a_{31}(15\ell^{2}) + a_{31}(7\ell^{2}) \equiv 0 \pmod{2}, \label{eq:a-15-7} \\
	a_{31}(23\ell^{2}) + a_{31}(15\ell^{2}) \equiv 0 \pmod{2}. \label{eq:a-23-15}
\end{gather}
The left-hand side of \eqref{eq:a-15-7} is the coefficient of $q^{\ell}$ in $\Shim_{15}R^{-31} + \Shim_{7}R^{-31}$. As before, we compute that
\[
	\Shim_{15}R^{-31} + \Shim_{7}R^{-31} \equiv F^{5} \pmod{2}.
\]
Since $\ell\equiv 1,7 \pmod{8}$ and the form $F^{5} \pmod{2}$ is supported on exponents which are $5 \pmod{8}$, congruence \eqref{eq:a-15-7} is true. Similarly, the left-hand side of \eqref{eq:a-23-15} is the coefficient of $q^{\ell}$ in $\Shim_{23}R^{-31} + \Shim_{15}R^{-31}$. We compute that
\[
	\Shim_{23}R^{-31} + \Shim_{15}R^{-31} \equiv F^{3} \pmod{2}.
\]
Since $\ell\equiv 1,7 \pmod{8}$ and the form $F^{3} \pmod{2}$ is supported only on exponents which are $3 \pmod{8}$, congruence \eqref{eq:a-23-15} is true.
\end{proof}
\end{section}

\bibliographystyle{plain}
\bibliography{bibliography.bib}

\end{document}